%% file: Main.tex
\documentclass[final,5p,times,10pt]{elsarticle}
\usepackage[centertags]{amsmath}
\usepackage{amsmath,amsfonts,amssymb,amsthm,epstopdf,newlfont,graphicx}
\usepackage{booktabs,bm,enumitem,setspace,empheq,cancel}
\usepackage{algorithm}
\usepackage[short,nocomma]{optidef}
\usepackage{algpseudocode}
\usepackage{caption,subcaption}
\usepackage{threeparttable,multirow}
\usepackage{amsmath,accents}
\usepackage[T1]{fontenc}
\usepackage[american]{babel}
\usepackage{natbib}
\biboptions{comma,square,sort&compress}
\usepackage{appendix}
\usepackage[colorlinks=true]{hyperref}

\newcommand\hcancel[2][black]{\setbox0=\hbox{$#2$}%
\rlap{\raisebox{.25\ht0}{\textcolor{#1}{\rule{0.7\wd0}{0.75pt}}}}#2} 
\newcommand\hcancelt[2][black]{\setbox0=\hbox{$#2$}%
\rlap{\raisebox{.25\ht0}{\textcolor{#1}{\hspace{0.3mm}\rule{0.7\wd0}{0.75pt}}}}#2} 
\newtheorem{thm}{Theorem}[section]

\newtheorem{cor}{Corollary}[section]

\theoremstyle{definition}

\numberwithin{algorithm}{section}
\numberwithin{equation}{section}
\renewcommand{\theequation}{\thesection.\arabic{equation}}
\def\simgt{\,\hbox{\lower0.6ex\hbox{$>$}\llap{\raise0.3ex\hbox{$\sim$}}}\,}
\def\simlt{\,\hbox{\lower0.6ex\hbox{$<$}\llap{\raise0.3ex\hbox{$\sim$}}}\,}
\def\simgteq{\,\hbox{\lower0.6ex\hbox{$\ge$}\llap{\raise0.6ex\hbox{$\sim$}}}\,}
\def\simlteq{\,\hbox{\lower0.6ex\hbox{$\le$}\llap{\raise0.6ex\hbox{$\sim$}}}\,}
\def\applteq{\,\hbox{\lower0.6ex\hbox{$\le$}\llap{\raise0.8ex\hbox{$\approx$}}}\,}
\def\applt{\,\hbox{\lower0.6ex\hbox{$<$}\llap{\raise0.5ex\hbox{$\approx$}}}\,}
\DeclareMathAlphabet\mathbfcal{OMS}{cmsy}{b}{n}

\makeatletter
\def\user@resume{resume}
\def\user@intermezzo{intermezzo}
\newcounter{previousequation}
\newcounter{lastsubequation}
\newcounter{savedparentequation}
\makeatother
\newcommand{\C}[1]{\mathcal{#1}}

\newcommand{\F}[1]{\mathbf{#1}}

\newcommand{\bsC}[1]{\boldsymbol{\C{#1}}}

\newcommand{\MB}[1]{\mathbb{#1}}
\newcommand{\MBS}{\MB{S}}
\newcommand{\MBG}{\MB{G}}
\newcommand{\MBSG}{\hat{\MBG}}
\newcommand{\MBR}{\mathbb{R}}

\newcommand{\MBRP}{\MBR^+}
\newcommand{\MBRzer}{\MBR_0}
\newcommand{\MBRzerP}{\MBRzer^+}
\newcommand{\MBZ}{\mathbb{Z}}
\newcommand{\MBZP}{\MBZ^+}
\newcommand{\MBZzer}{\MBZ_0}
\newcommand{\MBZzerP}{\MBZzer^+}
\newcommand{\MBZe}{\MBZ_e}
\newcommand{\MBZeP}{\MBZe^+}
\newcommand{\MBZzereP}{\MBZ_{0,e}^+}
\newcommand{\MBZOP}{\MBZ_{o}^+}
\newcommand{\MBT}{\mathbb{T}}
\newcommand{\MBJ}{\mathbb{J}}
\newcommand{\MBN}{\mathbb{N}}
\newcommand{\MFF}{\mathfrak{F}}
\newcommand{\MFP}{\mathfrak{P}}
\newcommand{\MFC}{\mathfrak{C}}

\newcommand{\MFX}{\mathfrak{X}}
\newcommand{\MFU}{\mathfrak{U}}

\newcommand{\cancbra}[1]{\hcancel{[}#1\hcancelt{]}}
\newcommand{\sumd}{\sideset{}{'}}
\newcommand{\bmc}{\bm{c}}

\newcommand{\bmx}{\bm{x}}

\newcommand{\bmt}{\bm{t}}

\newcommand{\bmz}{\bm{z}}

\newcommand{\bmy}{\bm{y}}

\newcommand{\bmh}{\bm{h}}

\newcommand{\bmu}{\bm{u}}

\newcommand{\bmf}{\bm{f}}

\newcommand{\bmzer}{\bm{\mathit{0}}}
\newcommand{\bmone}{\bm{\mathit{1}}}

\newcommand{\bmX}{\bm{X}}

\newcommand{\hbmz}{\hat\bmz}

\newcommand{\hz}{\hat{z}}

\newcommand{\FOmega}{\F{\Omega}}

\newcommand{\foralla}{\,\forall_{\mkern-6mu a}\,}
\newcommand{\foralle}{\,\forall_{\mkern-6mu e}\,}
\newcommand{\foralls}{\,\forall_{\mkern-6mu s}\,}

\newcommand{\Def}[1]{\text{Def}\left(#1\right)}

\newcommand{\GLD}[4]{{}_{\;\;#1}^{GL}D_{#2}^{#3}{#4}}
\newcommand{\RLD}[4]{{}_{\;\;#1}^{RL}D_{#2}^{#3}{#4}}
\newcommand{\CD}[4]{{}_{#1}^CD_{#2}^{#3}{#4}}
\newcommand{\MRLD}[4]{{}_{\hspace{3.6mm}#1}^{MRL}D_{#2}^{#3}{#4}}
\newcommand{\MCD}[4]{{}_{\;\;#1}^{MC}D_{#2}^{#3}{#4}}
\newcommand{\MD}[4]{{}_{#1}^{M}D_{#2}^{#3}{#4}}
\newcommand{\ED}[4]{{}_{#1}^{E}D_{#2}^{#3}{#4}}
\usepackage{mathtools}
\mathtoolsset{showonlyrefs}
\makeatletter
\def\BState{\State\hskip-\ALG@thistlm}
\makeatother
\MakeRobust{\eqref}
\errorcontextlines\maxdimen

\makeatletter
    \newcommand*{\algrule}[1][\algorithmicindent]{\makebox[#1][l]{\hspace*{.5em}\thealgruleextra\vrule height \thealgruleheight depth \thealgruledepth}}%
\newcommand*{\thealgruleextra}{}
\newcommand*{\thealgruleheight}{.75\baselineskip}
\newcommand*{\thealgruledepth}{.25\baselineskip}

\newcount\ALG@printindent@tempcnta
\def\ALG@printindent{%
    \ifnum \theALG@nested>0
        \ifx\ALG@text\ALG@x@notext
        \else
            \unskip
            \addvspace{-1pt}
            \ALG@printindent@tempcnta=1
            \loop
                \algrule[\csname ALG@ind@\the\ALG@printindent@tempcnta\endcsname]%
                \advance \ALG@printindent@tempcnta 1
            \ifnum \ALG@printindent@tempcnta<\numexpr\theALG@nested+1\relax
            \repeat
        \fi
    \fi
    }%
\usepackage{etoolbox}
\patchcmd{\ALG@doentity}{\noindent\hskip\ALG@tlm}{\ALG@printindent}{}{\errmessage{failed to patch}}
\makeatother

\newbox\statebox
\newcommand{\myState}[1]{%
    \setbox\statebox=\vbox{#1}%
    \edef\thealgruleheight{\dimexpr \the\ht\statebox+1pt\relax}%
    \edef\thealgruledepth{\dimexpr \the\dp\statebox+1pt\relax}%
    \ifdim\thealgruleheight<.75\baselineskip
        \def\thealgruleheight{\dimexpr .75\baselineskip+1pt\relax}%
    \fi
    \ifdim\thealgruledepth<.25\baselineskip
        \def\thealgruledepth{\dimexpr .25\baselineskip+1pt\relax}%
    \fi
    \State #1%
    \def\thealgruleheight{\dimexpr .75\baselineskip+1pt\relax}%
    \def\thealgruledepth{\dimexpr .25\baselineskip+1pt\relax}%
}
\makeatletter
\newcommand{\oset}[3][0ex]{%
  \mathrel{\mathop{#3}\limits^{
    \vbox to#1{\kern-2\ex@
    \hbox{$\scriptstyle#2$}\vss}}}}
\makeatother
\begin{document}
\begin{frontmatter}
\input{Tit}
\input{Abs}
\end{frontmatter}

\section{Introduction}
\label{Int}
The main goal of the class of periodic fractional optimal control problems (PFOCPs) that we first introduced in \cite{elgindy2023fourier} is to find periodic solutions to optimal control (OC) problems governed by fractional dynamic systems. This class of problems is expected to be very effective in modeling periodic phenomena and real-life problems with greater accuracy than periodic integer-order OC problems. We showed in \cite{elgindy2023fourier} how to derive the optimal states and controls numerically with excellent accuracy and rapid convergence through a stable scheme based on Fourier collocation and Gegenbauer quadrature. For a fractional order $\alpha \in (0, 1)$, we pointed out how to reduce the
integral defining the fractional derivative (FD) to an easy-to-evaluate one that includes only the first derivative of the trigonometric Lagrange interpolating polynomial associated with the Fourier interpolation of the periodic function. In this study, we extend this framework by introducing another magical change of variables formula that can reduce the integral defining the FD of any positive non-integer fractional order to an easy-to-evaluate one that includes only the $m$th-derivative of the trigonometric Lagrange interpolating polynomial, where $m$ is the ceiling value of $\alpha$. We show how to generalize the Fourier-Gegenbauer (FG)-based pseudospectral (PS) FGPS method first presented in \cite{elgindy2023fourier} to manage periodic higher-order fractional OC problems (PHFOCPs). We also show how to extend the notions of the $\alpha$th-order FGPS quadrature (FGPSQ) and its associated integration matrix (FGPSFIM) with index $L$, which proved to be very effective in computing periodic FDs, to exist for any positive non-integer fractional order $\alpha$. Another important contribution of this study is the derivation of more general and compact theorems on the truncation error bounds associated with the FGPS quadrature induced by the FGPSFIM, which highlights the quality of FGPS approximations to periodic FDs in light of some of the main parameters associated with the periodic FD and the proposed FGPS method. To the best of our knowledge, this paper introduces the first numerical approach in the literature to solve PHFOCPs of any positive non-integer fractional order.

The remainder of this study is structured as follows. In Section \ref{sec:PN}, we provide preliminaries and notations to simplify the presentation of this paper. The PHFOCP is introduced in general form in Section \ref{sec:PS1}. In Section \ref{sec:FGPMFD1}, we derive the extended FGPS formulas required to approximate the periodic FDs for any positive non-integer fractional order. Section \ref{sec:ECA1} presents error and convergence analyses of the derived FGPS formulas. The performance of the extended FGPS method is demonstrated in Section \ref{sec:NS1} followed by concluding remarks in Section \ref{sec:Conc}.

\section{Preliminaries and Notations}
\label{sec:PN}
The following notations are used throughout this study to abridge and simplify the mathematical formulas. 

\noindent\textbf{Logical Symbols.} The  symbols $\forall, \foralla, \foralle$, and $\foralls$ stand for the phrases ``for all,'' ``for any,'' ``for each,'' and ``for some,'' in respective order. $f \in \Def{\FOmega}$ means the function $f$ is defined on the set $\FOmega$.\\[0.5em]
\textbf{List and Set Notations.} $\MFC, \MFF$, and $\MFP$ denote the set of all complex-valued, real-valued, and piecewise continuous real-valued functions, respectively. Moreover, $\MBR$, $\MBZ, \MBZP, \MBZzerP, \MBZOP$, $\MBZeP$, and $\MBZzereP$ denote the sets of real numbers, integers, positive integers, non-negative integers, positive odd integers, positive even integers, and non-negative even integers, respectively. The notations $i$:$j$:$k$ or $i(j)k$ indicate a list of numbers from $i$ to $k$ with increment $j$ between numbers, unless the increment equals one where we use the simplified notation $i$:$k$. For example, $0$:$0.5$:$2$ simply means the list of numbers $0, 0.5, 1, 1.5$, and $2$, while $0$:$2$ means $0, 1$, and $2$. The list of symbols $y_1, y_2, \ldots, y_n$ is denoted by $\left. y_i \right|_{i=1:n}$ or simply $y_{1:n}$, and their set is represented by $\{y_{1:n}\}\,\foralla n \in \MBZP$. We define $\MBJ_n = \{0:n-1\}$ and $\MBN_n = \{1:n\}\,\foralla n \in \MBZP$. $\MBS_n^{T} = \left\{t_{n,0:n-1}\right\}$ is the set of $n$ equally-spaced points such that $t_{n,j} = T j/n\, \forall j \in \MBJ_n$. $\MBG_n^{\lambda} = \left\{z_{n,0:n}^{\lambda}\right\}$ is the set of Gegenbauer-Gauss (GG) zeros of the $(n+1)$st-degree Gegenbauer polynomial with index $\lambda > -1/2$, and $\MBSG_{1,n}^{\lambda} = \left\{\hz_{n,0:n}^{\lambda}: \hz_{n,0:n}^{\lambda} = \frac{1}{2} \left(z_{n,0:n}^{\lambda}+1\right)\right\}$ is the shifted Gegenbauer (SG)-Gauss (or SGG) points set in the interval $[0,1]\,\foralla n \in \MBZP$; cf. \cite{Elgindy201382,elgindy2018high,elgindy2018optimal}. Finally, the specific interval $[0, T]$ is denoted by $\FOmega_{T}\,\forall T > 0$; for example, $[0, t_{n,j}]$ is denoted by ${\FOmega_{t_{n,j}}}\,\forall j \in \MBJ_n$.\\[0.5em] 
\textbf{Function Notations.} $\delta_{n,m}$ is the usual Kronecker delta function of variables $n$ and $m$. $\left\lceil  \cdot  \right\rceil, \left\lfloor {\cdot} \right\rfloor$ and $\Gamma$ denote the ceil, floor, and Gamma functions, respectively. $\left( {\begin{array}{*{20}{c}}
\alpha \\
k
\end{array}} \right)$ is the binomial coefficient indexed by the pair $\alpha \in \MBR$ and $k \in \MBZzerP$. For convenience, we shall denote $g(t_{n})$ by $g_n \foralla g \in \MFC, n \in \MBZ, t_n \in \MBR$, unless stated otherwise.\\[0.5em]
\textbf{Integral Notations.} We denote $\int_0^{b} {h(t)\,dt}$ and $\int_a^{b} {h(t)\,dt}$ by $\C{I}_{b}^{(t)}h$ and $\C{I}_{a, b}^{(t)}h$, respectively, $\foralla$ integrable $h \in \MFC, \{a, b\} \subset \MBR$. If the integrand function $h$ is to be evaluated at any other expression of $t$, say $u(t)$, we express $\int_0^{b} {h(u(t))\,dt}$ and $\int_a^b {h(u(t))\,dt}$ with a stroke through the square brackets as $\C{I}_{b}^{(t)}h\cancbra{u(t)}$ and $\C{I}_{a,b}^{(t)}h\cancbra{u(t)}$ in respective order.\\[0.5em] 
\textbf{Space and Norm Notations.} $\MBT_{T}$ is the space of $T$-periodic, univariate functions $\foralla T \in \MBRP$. $C^k(\FOmega)$ is the space of $k$ times continuously differentiable functions on ${\FOmega}\,\forall k \in \MBZzerP$. ${}_T\mathfrak{X}_{n_1}^{n_2} = \{[y_0, \ldots, y_{n_1}]^{\top}: \MBRzerP \to \MBR^{n_1}\text{ s.t. }y_j \in \MBT_{T} \cap C^{n_2}(\MBRzerP) \forall j \in \MBN_{n_1}\}$ is the space of $T$-periodic, $n_2$-times continuously differentiable, $n_1$-dimensional vector functions on $\MBRzerP$. ${}_T\MFU_{n} = \{[u_0, \ldots, u_{n}]^{\top}: \MBRzerP \to \MBR^{n}\text{ s.t. }u_j \in \MBT_{T} \cap \MFP\,\forall j$ $\in \MBN_{n}\}$ is the space of $T$-periodic, $n$-dimensional piecewise continuous vector functions on $\MBRzerP$. $L^p({\FOmega})$ is the Banach space of measurable functions $u$ defined on ${\FOmega}$ such that ${\left\| u \right\|_{{L^p}}} = {\left( {{\C{I}_{\FOmega}}{{\left| u \right|}^p}} \right)^{1/p}} < \infty\,\forall p \ge 1$. In particular, we write $\left\|u\right\|_{\infty}$ to denote ${\left\| u \right\|_{{L^{\infty}}}}$. Finally, $\left\|\cdot\right\|_2$ denotes the usual Euclidean norm of vectors.\\[0.5em]
\textbf{Vector Notations.} We shall use the shorthand notations $\bmt_N$ and $g_{0:N-1}$ to stand for the column vectors $[t_{0}, t_{1}, \ldots$, $t_{N-1}]^{\top}$ and $[g_0, g_1, \ldots, g_{N-1}]^{\top}\,\forall N \in \MBZP$ in respective order. In general, $\foralla h \in \MFC$ and vector $\bmy$ whose $i$th-element is $y_i \in \MBR$, the notation $h(\bmy)$ stands for a vector of the same size and structure of $\bmy$ such that $h(y_i)$ is the $i$th element of $h(\bmy)$. Moreover, by $\bmh(\bmy)$ or $h_{1:m}\cancbra{\bmy}$ with a stroke through the square brackets, we mean $[h_1(\bmy), \ldots, h_m(\bmy)]^{\top}\,\foralla m$-dimensional column vector function $\bmh$, with the realization that the definition of each array $h_i(\bmy)$ follows the former notation rule $\foralle i$. Furthermore, if $\bmy$ is a vector function, say $\bmy = \bmy(t)$, then we write $h(\bmy(\bmt_N))$ and $\bmh(\bmy(\bmt_N))$ to denote $[h(\bmy(t_0)), h(\bmy(t_1)), \ldots, h(\bmy(t_{N-1}))]^{\top}$ and $[\bmh(\bmy(t_0)), \bmh(\bmy(t_1)), \ldots, \bmh(\bmy(t_{N-1}))]^{\top}$ in respective order.\\[0.5em] 
\textbf{Matrix Notations.} $\F{O}_n, \F{1}_n$, and $\F{I}_n$ stand for the zero, all ones, and the identity matrices of size $n$. $\F{C}_{n,m}$ indicates that $\F{C}$ is a rectangular matrix of size $n \times m$; moreover, $\F{C}_n$ denotes a row vector whose elements are the $n$th-row elements of $\F{C}$, except when $\F{C}_n = \F{O}_n, \F{1}_n$, or $\F{I}_n$, where it denotes the size of the matrix. For convenience, a vector is represented in print by a bold italicized symbol while a two-dimensional matrix is represented by a bold symbol, except for a row vector whose elements form a certain row of a matrix where we represent it in bold symbol as stated earlier. For example, $\bmone_n$ and $\bmzer_n$ denote the $n$-dimensional all ones- and zeros- column vectors, while $\F{1}_n$ and $\F{O}_n$ denote the all ones- and zeros- matrices of size $n$, respectively. Finally, the notation $[.;.]$ denotes the usual vertical concatenation.\\[0.5em]
\textbf{Common Fractional Differentiation Formulas.} Let $\alpha \in \MBRP$, $m = \left\lceil  \alpha  \right\rceil$, and $f \in \Def{\FOmega_T}$. The $\alpha$-th order Gr\"{u}nwald-Letnikov derivative of $f$ with respect to $t$ and a terminal value $a$ is given by
\begin{equation}
\GLD{a}{t}{\alpha}{f(t)} = \mathop {\lim }\limits_{\scriptstyle h \to 0\atop
\scriptstyle nh = t - a} {h^{ - \alpha }}\sum\limits_{k = 0}^n {{{( - 1)}^k}\left( {\begin{array}{*{20}{c}}
\alpha \\
k
\end{array}} \right)f(t - kh)}.
\end{equation}
The $\alpha$-th order left RL and Caputo FDs are denoted by $\RLD{0}{t}{\alpha}{f(t)}$ and $\CD{0}{t}{\alpha}{f(t)}$, respectively, and are defined for $t \in \FOmega_T$ by
\begin{align}
\RLD{0}{t}{\alpha}{f(t)} &= \left\{ \begin{array}{l}
\frac{1}{{\Gamma (m - \alpha )}}\frac{{{d^m}}}{{d{t^m}}}\C{I}_{t}^{(\tau)}\left[{{{(t - \tau )}^{m - \alpha  - 1}}f}\right],\quad \alpha \notin \MBZP,\\
f^{(m)}(t),\quad \alpha \in \MBZP,
\end{array} \right.\\
\CD{0}{t}{\alpha}{f(t)} &= \left\{ \begin{array}{l}
\frac{1}{{\Gamma (m - \alpha )}} \C{I}_t^{(\tau)} \left[{{{(t - \tau )}^{m - \alpha  - 1}}f^{(m)}}\right],\quad \alpha \notin \MBZP,\\
f^{(m)}(t),\quad \alpha \in \MBZP.
\end{array} \right.
\end{align}
The RL and Caputo FDs with sliding fixed memory length $L > 0$, and denoted by $\MRLD{L}{t}{\alpha}{f(t)}$ and $\MCD{L}{t}{\alpha}{f(t)}$, respectively, are defined by
\begin{align}
\MRLD{L}{t}{\alpha}{f(t)} = \left\{ \begin{array}{l}
\frac{1}{{\Gamma (m - \alpha )}}\frac{{{d^m}}}{{d{t^m}}}\C{I}_{t-L, t}^{(\tau)}\left[{{{(t - \tau )}^{m - \alpha  - 1}}f}\right],\quad \alpha \notin \MBZP,\\
f^{(m)}(t),\quad \alpha \in \MBZP,
\end{array} \right.\\
\MCD{L}{t}{\alpha}{f(t)} = \left\{ \begin{array}{l}
\frac{1}{{\Gamma (m - \alpha )}} \C{I}_{t-L, t}^{(\tau)}\left[{{{(t - \tau )}^{m - \alpha  - 1}}f^{(m)}}\right],\quad \alpha \notin \MBZP,\label{eq:PMFCD1}\\
f^{(m)}(t),\quad \alpha \in \MBZP.
\end{array} \right.
\end{align}
If $f \in C^{(m)}(\MBRzerP)$, then $\MRLD{L}{t}{\alpha}{f(t)} = \MCD{L}{t}{\alpha}{f(t)}$, so we can denote both modified fractional operators by $\MD{L}{t}{\alpha}{}$. A reduced form of $\MD{L}{t}{\alpha}{}$ with constant integration limits, denoted by $\ED{L}{t}{\alpha}{f(t)}$, is given by
\begin{equation}\label{eq:RedElgPMFCD1}
\ED{L}{t}{\alpha}{f(t)} = \frac{L^{m-\alpha}}{(1-\alpha) \Gamma(m-\alpha)} \C{I}_1^{(y)} {\left[y^{\frac{m-1}{1-\alpha}} f^{(m)}\cancbra{t-L\,y^{\frac{1}{1-\alpha}}}\right]},
\end{equation}
which simplifies into
\begin{equation}\label{eq:RedElgPMFCD12}
\ED{L}{t}{\alpha}{f(t)} = \frac{L^{1-\alpha}}{\Gamma(2-\alpha)} \C{I}_1^{(y)} {f'\cancbra{t-L\,y^{\frac{1}{1-\alpha}}}}\quad \foralla \alpha \in (0,1),
\end{equation}
cf. \cite{elgindy2023fourier}.

\section{Problem Statement}
\label{sec:PS1}
In this section, we consider the general form of PHFOCPs governed by $\MD{L}{t}{\alpha}{}$-based FDEs of any order $\alpha \in \MBRP\backslash\MBZP$. In particular, our aim is to approximate the optimal $T$-periodic solutions of the following PHFOCP:
\begin{mini}
  {\bmu}{J(\bmu) = \frac{1}{T} \C{I}_{T}^{(t)} {g}\cancbra{\bmx(t), \bmu(t), t}}{}{}
  {\label{eq:OC1}}{}
  \addConstraint{\MD{L}{t}{\alpha}{\bmx(t)}}{= \bmf\left(\bmx(t),\bmu(t),t\right)}{\quad \forall t \in {\FOmega_T},}
  \addConstraint{\bmc(\bmx(t)}{, \bmu(t), t) \le \bmzer_p}{\quad \forall t \in {\FOmega_T},}
\end{mini}
where $p \in \MBZP, \bmx \in {}_T\MFX_{n_x}^m, \bmu \in {}_T\MFU_{n_u}, g: \MBR^{n_x} \times \MBR^{n_u} \times \MBRzerP \to \MBR, \bmf = [f_1, \ldots, f_n]^{\top}: \MBR^{n_x} \times \MBR^{n_u} \times \MBRzerP \to \MBR^{n_x}$, and $\bmc = [c_1, \ldots, c_p]^{\top}: \MBR^{n_x} \times \MBR^{n_u} \times \MBRzerP \to \MBR^p$ with $\{g, f_i, c_i\} \subset C^{k}(\MBRzerP)\,\foralls k \in \MBZzerP$. Here, $\bmx$ and $\bmu$ are the state and control vector functions, respectively. We assume that a solution of the problem exists.

\section{The FGPS Approximation of \texorpdfstring{$\MD{L}{t}{\alpha}{}$}{The FGPS Approximation of the RL and Caputo FDs With Sliding Fixed Memory Length}}
\label{sec:FGPMFD1}
To generalize the recent work of \citet{elgindy2023fourier}, we introduce the following $m$-dependent change of variables equation:
\begin{equation}\label{eq:Elg1}
\tau = t-L\,y^{\frac{1}{m-\alpha}},
\end{equation}
to transform $\MD{L}{t}{\alpha}{f(t)}$ defined by Eq. \eqref{eq:PMFCD1} into a reduced form, denoted by $\ED{L}{t}{\alpha}{f(t)}$, and defined by
\begin{equation}\label{eq:RedElgPMFCD1}
\ED{L}{t}{\alpha}{f(t)} = \frac{L^{m-\alpha}}{\Gamma(m-\alpha+1)} \C{I}_1^{(y)} {f^{(m)}\cancbra{t-L\,y^{\frac{1}{m-\alpha}}}},
\end{equation}
$\forall t \in \FOmega_T, \alpha \in \MBRP\backslash\MBZP$. Notice that for $0 < \alpha < 1$, Eq. \eqref{eq:RedElgPMFCD1} simplifies into \eqref{eq:RedElgPMFCD12}. Let 
\begin{equation}\label{eq:eqLF1}
{I_N}f(t) = \sum\limits_{j = 0}^{N - 1} {{f_j}{\C{F}_j}(t)},
\end{equation}
be the $N/2$-degree, $T$-periodic Fourier interpolant that matches a function $f \in \MBT_T$ at the set of nodes $\MBS_N^T$, where ${\C{F}_j}(t)$ is the $N/2$-degree trigonometric Lagrange interpolating polynomial given by
\begin{equation}\label{eq:CFL2}
\C{F}_j(t) = \frac{1}{N}\sumd\sum\limits_{\left| k \right| \le N/2} {\cos\left({\omega _k}(t - {t_{N,j}})\right)}\quad \foralla f \in \MFF,
\end{equation}
where $\omega_a = 2 \pi a/T\, \forall a \in \MBR$, and the primed sigma denotes a summation in which the last term is omitted \cite{elgindy2019high}. Substituting Eq. \eqref{eq:CFL2} into Eq. \eqref{eq:eqLF1} yields the following approximation:
\begin{equation}\label{eq:RedElgPMFCD101app1}
\ED{L}{t}{\alpha}{f(t)} \approx \ED{L}{t}{\alpha}{I_Nf(t)} = \frac{L^{m-\alpha}}{\Gamma(m-\alpha+1)} \sum\limits_{j = 0}^{N - 1} {{f_j} \C{I}_1^{(y)}{{\C{F}^{(m)}_j}\cancbra{t-L\,y^{\frac{1}{m-\alpha}}}}},
\end{equation}
where the $m$th-derivative of ${\C{F}_j}$ is given by\\
\scalebox{0.9}{\parbox{\linewidth}{%
\begin{equation}\label{eq:CFL2nth1}
\C{F}_j^{(m)}(t) = (-1)^{\left\lfloor {\frac{{m + 1}}{2}} \right\rfloor} \frac{1}{N} \left(\frac{2\pi}{T}\right)^{m} \sumd\sum\limits_{\scriptstyle\left| k \right| \le N/2\atop
\scriptstyle k \ne 0} {k^{m} \sin\left({\omega _k}(t - {t_{N,j}}) + {\delta _{\frac{{m}}{2},\left\lfloor {\frac{{m}}{2}} \right\rfloor }}\frac{\pi }{2}\right)};
\end{equation}  
}}\\
cf. \cite{elgindy2023fourier}. We can approximate $\C{I}_1^{(y)}{{\C{F}^{(m)}_j}\cancbra{t-L\,y^{\frac{1}{m-\alpha}}}}$ at any mesh point $t_{N,l} \in \MBS_{N}^T$ using the SG quadratures; cf. \cite{Elgindy20161,Elgindy20171,elgindy2018optimal,Elgindy2023a}, to obtain the following $\alpha$th-order FGPS quadrature (FGPSQ) with index $L$:
\begin{gather}
\C{I}_1^{(y)}{{\C{F}^{(m)}_j}\cancbra{t_{N,l}-L\,y^{\frac{1}{m-\alpha}}}} \approx {}_{L}^E\C{Q}^{\alpha}_{N_G,l,j}\nonumber\\
= \frac{1}{2} \left[\F{P}\,{\C{F}^{(m)}_j}\left(t_{N,l}\,\bmone_{N_G+1}-L\,\left(\hbmz_{N_G+1}^{\lambda}\right)^{\frac{1}{m-\alpha}}\right)\right],\label{eq:Quaderr1}
\end{gather}
where $\F{P}$ is the $(N_G+1)$ dimensional, GG points-based integration row vector constructed using \cite[Algorithm 6 or 7]{Elgindy20171}. Substituting Eq. \eqref{eq:Quaderr1} into \eqref{eq:RedElgPMFCD101app1} yields the following general approximation formula for the FD of the $T$-periodic function $f$ at any mesh point:
\begin{equation}\label{eq:RedElgPMFCD101app2}
\ED{L}{t_{N,l}}{\alpha}{f(t)} \approx \frac{L^{m-\alpha}}{\Gamma(m-\alpha+1)} \sum\limits_{j = 0}^{N - 1} {{}_{L}^E\C{Q}^{\alpha}_{N_G,l,j} {f_j}}\quad \forall l \in \MBJ_{N},
\end{equation}
or in matrix notation,
\begin{equation}\label{eq:RedElgPMFCD101app3}
\ED{L}{\bmt_{N}}{\alpha}{f(t)} \approx \frac{L^{m-\alpha}}{\Gamma(m-\alpha+1)} \left({}_{L}^E\bsC{Q}_{N_G}^{\alpha}\,f_{0:N-1}\right),
\end{equation}  
where $\ED{L}{\bmt_{N}}{\alpha}{f(t)} = \left[\ED{L}{t_{N,0}}{\alpha}{f(t)}, \ldots, \ED{L}{t_{N,N-1}}{\alpha}{f(t)}\right]^{\top}$, and the $\alpha$th-order FGPS integration matrix (FGPSFIM) with index $L, {}_{L}^E\bsC{Q}^{\alpha}_{N_G}$, is given by 
\begin{align*}
&{}_{L}^E\bsC{Q}^{\alpha}_{N_G} = \left[{}_{L}^E\bsC{Q}^{\alpha}_{N_G,0}; \ldots; {}_{L}^E\bsC{Q}^{\alpha}_{N_G,N-1}\right]:\\
&{}_{L}^E\bsC{Q}^{\alpha}_{N_G,l} = \left[{}_{L}^E\C{Q}^{\alpha}_{N_G,l,0}, \ldots, {}_{L}^E\C{Q}^{\alpha}_{N_G,l,N-1}\right]\quad \forall l \in \MBJ_N.
\end{align*}
Figures \ref{fig:1}--\ref{fig:4} show the excellent approximations of the derived general FGPS formulas \eqref{eq:RedElgPMFCD101app2} and \eqref{eq:RedElgPMFCD101app3} to $\ED{L}{t}{\alpha}{\sin(t)}$ for the range $\alpha =$ 1.1:0.2:1.9, 1.99 using the parameter values $N \in \{4, 12, 40, 100\}$, $L = 30,  N_G = 1000$, and $\lambda = 0$. We observe that the graph of $\ED{L}{t}{\alpha}{\sin(t)}$ converges to the graph of the negative sine function as $\alpha \to 2$. 

It is interesting to find that the FGPS introduced in \cite{elgindy2023fourier} still applies for solving Problem \eqref{eq:OC1} with the replacement of the numerical differentiation formula of the state variables
\begin{equation}\label{eq:NDF1}
\ED{L}{\bmt_N}{\alpha}x_j(t) \approx \frac{L^{1-\alpha}}{\Gamma(2-\alpha)} \left({}_{L}^E\bsC{Q}_{N_G}^{\alpha}\,x_j(\bmt_N)\right)\quad \forall j \in \MBN_{n_x},
\end{equation}
with
\begin{equation}\label{eq:NDF1}
\ED{L}{\bmt_N}{\alpha}x_j(t) \approx \frac{L^{m-\alpha}}{\Gamma(m-\alpha+1)} \left({}_{L}^E\bsC{Q}_{N_G}^{\alpha}\,x_j(\bmt_N)\right)\quad \forall j \in \MBN_{n_x}.
\end{equation}

\begin{figure}[t]
\centering
\includegraphics[scale=0.5]{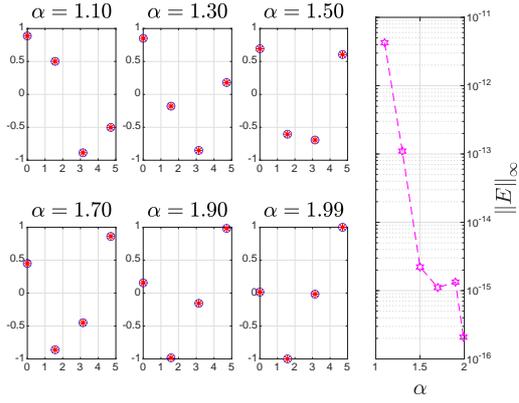}
\caption{The exact and FGPS approximate values of $\ED{L}{t}{\alpha}{\sin(t)}$ (Columns 1--3) and their corresponding maximum absolute errors (the 4th Column) for $\alpha = 1.1:0.2:1.9, 1.99$. The FGPS approximations were obtained using $N = 4, L = 30,  N_G = 1000$, and $\lambda = 0$. The exact values are shown in red color with * marker symbol, while the FGPS approximations are shown in blue colors with o marker symbol.}
\label{fig:1}
\end{figure}

\begin{figure}[t]
\centering
\includegraphics[scale=0.5]{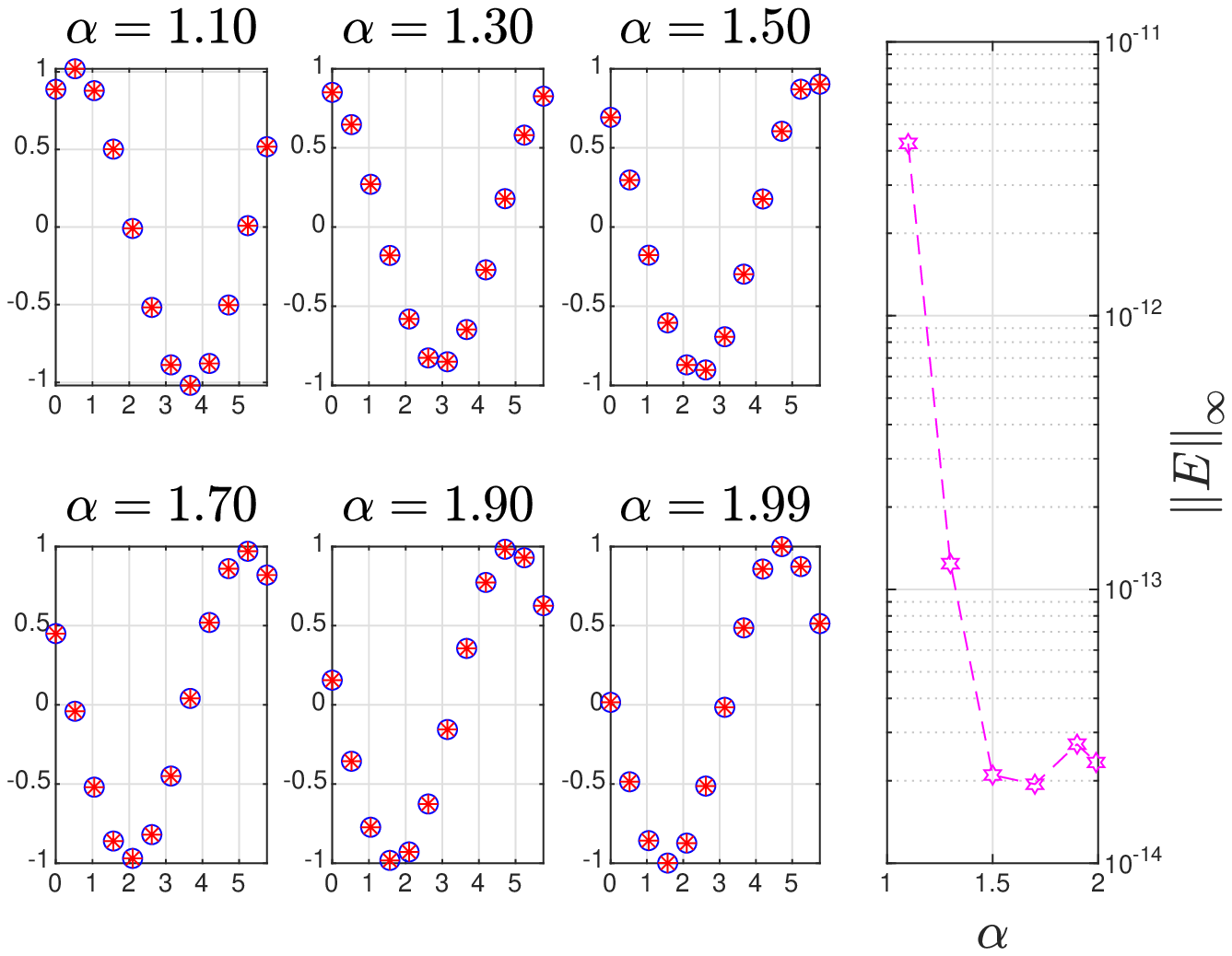}
\caption{The exact and FGPS approximate values of $\ED{L}{t}{\alpha}{\sin(t)}$ (Columns 1--3) and their corresponding maximum absolute errors (the 4th Column) for $\alpha = 1.1:0.2:1.9, 1.99$. The FGPS approximations were obtained using $N = 12, L = 30,  N_G = 1000$, and $\lambda = 0$. The exact values are shown in red color with * marker symbol, while the FGPS approximations are shown in blue colors with o marker symbol.}
\label{fig:2}
\end{figure}

\begin{figure}[t]
\centering
\includegraphics[scale=0.5]{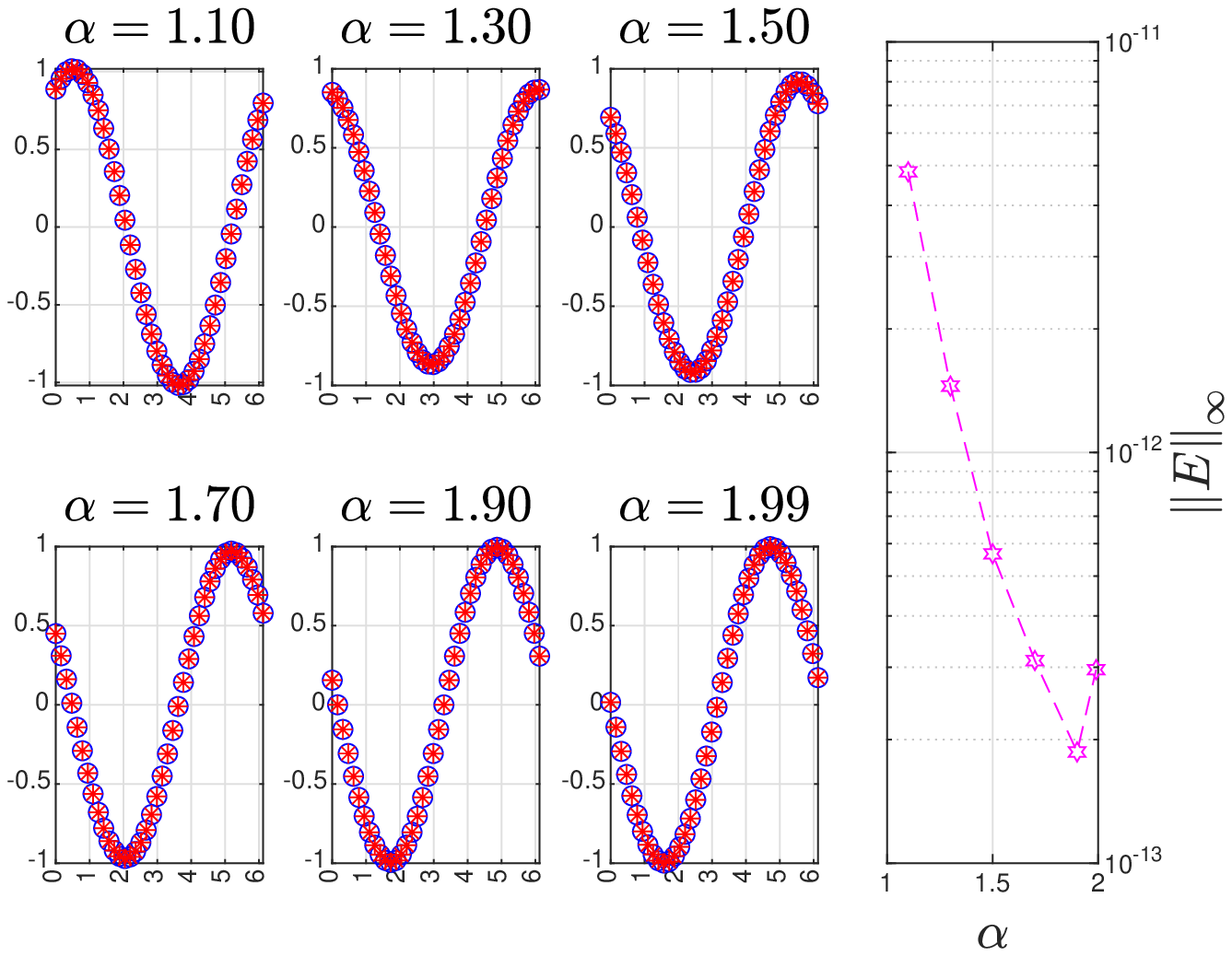}
\caption{The exact and FGPS approximate values of $\ED{L}{t}{\alpha}{\sin(t)}$ (Columns 1--3) and their corresponding maximum absolute errors (the 4th Column) for $\alpha = 1.1:0.2:1.9, 1.99$. The FGPS approximations were obtained using $N = 40, L = 30,  N_G = 1000$, and $\lambda = 0$. The exact values are shown in red color with * marker symbol, while the FGPS approximations are shown in blue colors with o marker symbol.}
\label{fig:3}
\end{figure}

\begin{figure}[t]
\centering
\includegraphics[scale=0.5]{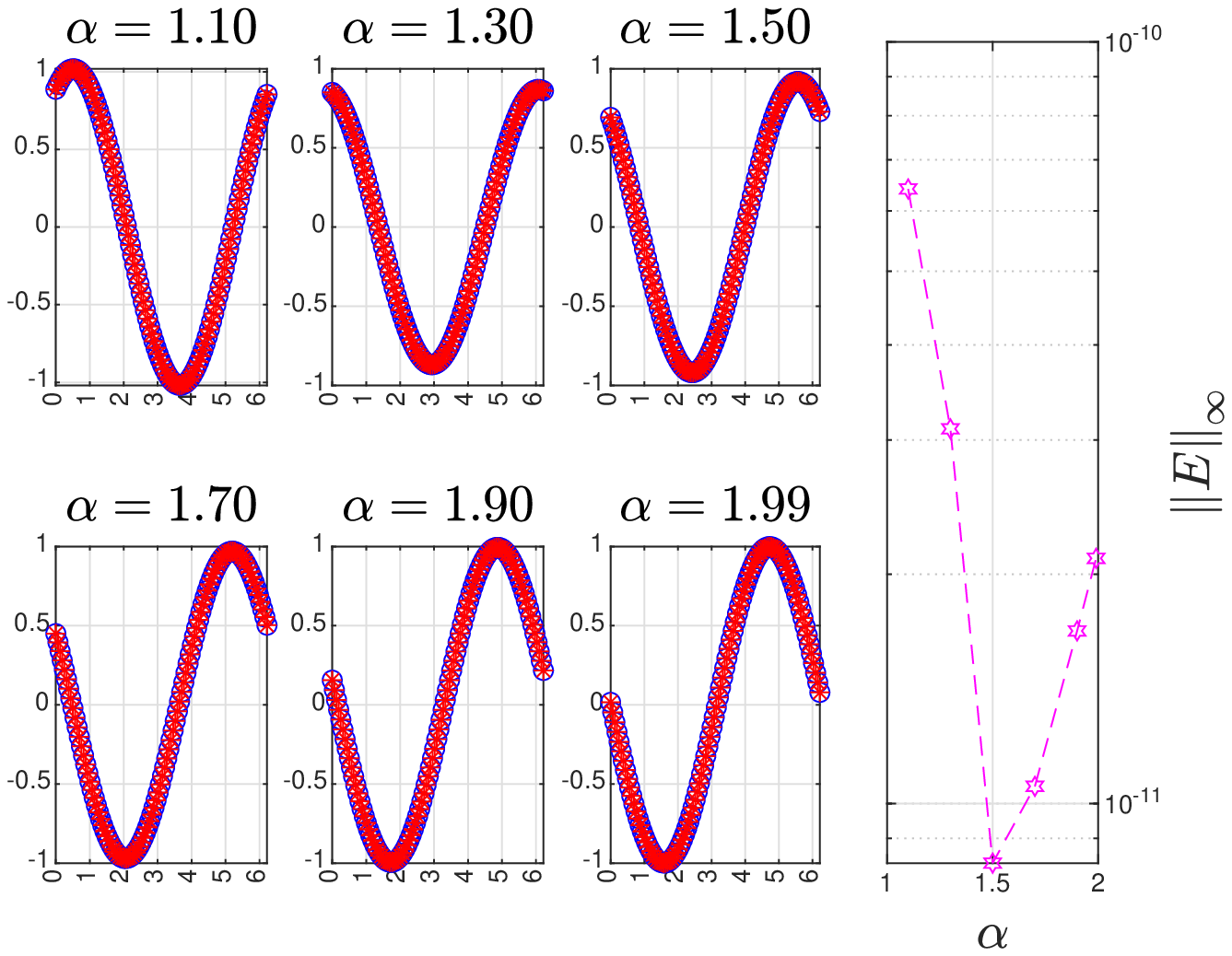}
\caption{The exact and FGPS approximate values of $\ED{L}{t}{\alpha}{\sin(t)}$ (Columns 1--3) and their corresponding maximum absolute errors (the 4th Column) for $\alpha = 1.1:0.2:1.9, 1.99$. The FGPS approximations were obtained using $N = 100, L = 30,  N_G = 1000$, and $\lambda = 0$. The exact values are shown in red color with * marker symbol, while the FGPS approximations are shown in blue colors with o marker symbol.}
\label{fig:4}
\end{figure}

\section{Error and Convergence Analysis}
\label{sec:ECA1}
The following theorem provides the general truncation error form of the FGPSQ \eqref{eq:Quaderr1}.
\begin{thm}\label{sec:erranalysgp1}
Suppose that ${\C{F}^{(m)}_j}\left(t-L\,y^{\frac{1}{m-\alpha}}\right)\,\forall j \in \MBJ_N$ is approximated by the SGG interpolant obtained through interpolation at the SGG points set $\MBSG_{1,N_G}^{\lambda}\,\foralls N_G \in \MBZP, t \in \FOmega_T, L \in \MBRP$; cf. \cite{Elgindy20161,Elgindy2023a}. Then $\exists\,\{\zeta_{0:N-1}\} \subset (0,1)$ such that the error in the FGPSQ \eqref{eq:Quaderr1}, denoted by $E_{N,N_G,j}^{\lambda,\alpha}$, is given by
\begin{equation}
E_{N,N_G,j}^{\lambda,\alpha}(\zeta_j; t) = \frac{\psi_{L,N_G,j}^{\alpha}(\zeta_j; t)}{{({N_G} + 1)!{\mkern 1mu} K_{{N_G} + 1}^{\left( {\lambda} \right)}}} \C{I}_1^{(y)} {\hat G_{{N_G} + 1}^{\left( {\lambda} \right)}},
\end{equation}
where
\begin{equation}\label{eq:Genius1}
\psi_{L,N_G,j}^{\alpha}(y;t) = \left(\frac{L}{\alpha-m} y^{-\frac{m-\alpha-1}{m-\alpha}}\right)^{N_G+1} \C{F}_j^{(N_G+m+1)}\left(t-L\,y^{\frac{1}{m-\alpha}}\right),
\end{equation}
and
\begin{equation}
K_{l}^{(\lambda )} = 2^{2l - 1} \frac{{\Gamma \left( {2\lambda  + 1} \right)\Gamma \left( {l + \lambda } \right)}}{{\Gamma \left( {\lambda  + 1} \right)\Gamma \left( {l + 2\lambda } \right)}}\quad \forall l \in \mathbb{Z}_0^+,
\end{equation}
is the leading coefficient of the $l$th-degree SGG polynomial with index $\lambda, {\hat G_l^{\left( {\lambda} \right)}}$.
\end{thm}
\begin{proof}
By the change of variables \eqref{eq:Elg1}, we have $y = [(t-\tau)/L]^{m-\alpha}$. The Chain Rule therefore yields
\begin{align}
\frac{{{d^{{N_G} + 1}}y}}{{d{y^{{N_G} + 1}}}}\C{F}_j^{(m)}(\tau ) &= {\left[ {\frac{{{L^{m - \alpha }}}}{{\alpha  - m}}{{(t - \tau )}^{ - m + \alpha  + 1}}} \right]^{{N_G} + 1}}\C{F}_j^{({N_G} + m + 1)}(\tau )\\
&=  \left(\frac{L}{\alpha-m} y^{-\frac{m-\alpha-1}{m-\alpha}}\right)^{N_G+1} \C{F}_j^{(N_G+m+1)}\left(t-L\,y^{\frac{1}{m-\alpha}}\right).\label{eq:pr1}
\end{align}
The proof is established by combining Eq. \eqref{eq:pr1} with \cite[Theorem 4.1]{Elgindy20161}. 
\end{proof}
The following corollary gives a more compact form of the truncation error formula derived earlier in \cite[Theorem 5.1]{elgindy2023fourier} for $0 < \alpha < 1$.
\begin{cor}\label{sec:erranalysgp1}
Let $0 < \alpha < 1$ and the assumption of Theorem \ref{sec:erranalysgp1} holds true. Then $\exists\,\{\zeta_{0:N-1}\} \subset (0,1)$ such that 
\begin{equation}
E_{N,N_G,j}^{\lambda,\alpha}(\zeta_j; t) = \frac{\psi_{L,N_G,j}^{\alpha}(\zeta_j; t)}{{({N_G} + 1)!{\mkern 1mu} K_{{N_G} + 1}^{\left( {\lambda} \right)}}} \C{I}_1^{(y)} {\hat G_{{N_G} + 1}^{\left( {\lambda} \right)}},
\end{equation}
where
\begin{equation}\label{eq:Genius1}
\psi_{L,N_G,j}^{\alpha}(y;t) = \left(\frac{L}{\alpha-1} y^{\frac{\alpha}{1-\alpha}}\right)^{N_G+1} \C{F}_j^{(N_G+2)}\left(t-L\,y^{\frac{1}{1-\alpha}}\right).
\end{equation}
\end{cor}
The following theorem underlines the truncation error bound of the FGPSQ \eqref{eq:Quaderr1}.
\begin{thm}\label{thm:Jan212022}
Suppose that the assumption of Theorem \ref{sec:erranalysgp1} holds true. Then there exist some constants $D^{\lambda} > 0$ and $B_1^{\lambda} > 1$, which depend only on $\lambda$, such that $\forall j \in \MBJ_N$,\\ 
\scalebox{0.9}{\parbox{\linewidth}{%
\begin{align}
	&{\left| E_{N,N_G,j}^{\lambda,\alpha}(\zeta_j; t)\right|} \le  
	D^{\lambda} N^{N_G+m} \left(\frac{L}{m-\alpha} \zeta_j^{-\frac{m-\alpha-1}{m-\alpha}}\right)^{N_G+1}\,{2^{ - 2{N_G} - 1}}{{{e}}^{{N_G}}}{N_G}^{\lambda - {N_G} - \frac{3}{2}} \times \nonumber\\
	&{\left\{ \begin{array}{l}
	1,\quad {N_G} \ge 0 \wedge \lambda \ge 0,\\
	\displaystyle{\frac{{\Gamma \left( {\frac{{{N_G}}}{2} + 1} \right)\Gamma \left( {\lambda + \frac{1}{2}} \right)}}{{\sqrt \pi\,\Gamma \left( {\frac{{{N_G}}}{2} + \lambda + 1} \right)}}},\quad N_G \in \MBZOP \wedge  - \frac{1}{2} < \lambda < 0,\\
	\displaystyle{\frac{{2\Gamma \left( {\frac{{{N_G} + 3}}{2}} \right)\Gamma \left( {\lambda + \frac{1}{2}} \right)}}{{\sqrt \pi  \sqrt {\left( {{N_G} + 1} \right)\left( {{N_G} + 2\lambda + 1} \right)}\,\Gamma \left( {\frac{{{N_G} + 1}}{2} + \lambda} \right)}}},\quad N_G \in \MBZzereP \wedge  - \frac{1}{2} < \lambda < 0,\\
	B_1^{\lambda} {\left( {{N_G} + 1} \right)^{ - \lambda}},\quad {N_G} \to \infty  \wedge  - \frac{1}{2} < \lambda < 0.
	\end{array} \right.}\label{eq:wow1}
\end{align}
}}\\\vspace{-4mm}
\end{thm}
\begin{proof}
Since 
\begin{equation}\label{eq:LFOrder1}
\C{F}_j^{(n+1)}(t) = O\left(N^n\right),\quad\text{ as }N \to \infty\,\forall (t,j,n) \in \FOmega_T \times \MBJ_N \times \MBZP;
\end{equation}
cf. \cite[Theorem 5.2]{elgindy2023fourier}, then 
\begin{equation}\label{eq:soly1}
\left|\psi_{L,N_G,j}^{\alpha}(\zeta_j; t)\right| \le c N^{N_G+m} \left(\frac{L}{m-\alpha} \zeta_j^{-\frac{m-\alpha-1}{m-\alpha}}\right)^{N_G+1}\quad \foralls c \in \MBRP.
\end{equation}
The rest of the proof follows by \cite[Theorem 5.5]{Elgindy2023a}.
\end{proof}
Theorem \ref{thm:Jan212022} proves the exponential convergence of the FGPS expansion for sufficiently smooth periodic functions as $N_G \to \infty$. We can also generally infer from the theorem that while holding the other parameters fixed, the larger the memory length, $L$, or the degree of Fourier interpolant, $I_Nf$, or the fractional order ceiling value, $m$, the larger the truncation error bound, and the larger the expected FD error.

\section{Numerical Simulations}
\label{sec:NS1}
Let $\bmx = [y_{1:2}]^{\top} \in {}_{\pi}\MFX_{2}^m,  \bmu = u \in {}_{\pi}\MFU_1$, and consider Problem \eqref{eq:OC1} with ${g}(\bmx(t), \bmu(t), t) = u^2-y_1^2, \bmf\left(\bmx(t),\bmu(t),t\right) = [y_2, -4 y_1 - 0.3 y_2 + u]^{\top}, -5 \le y_i \le 5\,\forall i \in \MBN_2$, and $-1 \le u(t) \le 1$. For $\alpha = 1$, this problem is a proper periodic OC problem in the sense that an optimal periodic admissible pair $(\bmx, \bmu)$ gives a strictly lower optimal performance index than that obtained by an optimal constant admissible pair; cf. \cite{gaitsgory2006linear,azzato2008applying}. \citet{gaitsgory2006linear} solved this problem earlier by solving an approximate finite-dimensional linear programming problem instead, and they estimated that if $J(\bmu) \approx -1.327\,\foralls$ admissible $T$-periodic pair $(\bmx, \bmu)$, then this pair is an approximate optimal solution. We solved this problem using the FGPS method performed using MATLAB R2023a software installed on a personal laptop equipped with a 2.9 GHz AMD Ryzen 7 4800H CPU and 16 GB memory running on a 64-bit Windows 11 operating system. We used MATLAB fmincon solver to solve the reduced nonlinear programming problem (NLP) with initial guesses of all ones and terminated the numerical optimization procedure whenever
\[\left\| {{\bmX^{(k + 1)}} - {\bmX^{(k)}}} \right\|_2 < {10^{ - 15}}\quad \text{or}\quad \left\|J_N^{(k+1)} - J_N^{(k)} \right\|_2 < {10^{ - 15}},\]
where $\bmX^{(k)} = [\bmx^{(k)}; \bmu^{(k)}]$ and $J_N^{(k)}$ denote the concatenated vector of approximate NLP minimizers and optimal cost function value at the $k$th iteration, respectively. The quality of the approximations were measured using the absolute discrete feasibility error (ADFE) at the collocation points; cf. \cite{Elgindy2023b,elgindy2023fourier}, denoted by $\bm{\C{E}}_N = (\C{E}_i)_{0 \le i \le N n_x-1}$. Figure \ref{fig:Fig5} shows the approximate optimal states and control variables obtained by the proposed method for $\alpha = 0.99999$ and $L = 30$. The approximate optimal performance index value was $J_{100} \approx -1.311$, rounded to three decimal digits with negligible ADFEs approaching the machine epsilon. Since the accuracy established by the FGPS approximations to approximate $\ED{L}{t}{\alpha}{\sin(t)}$ was roughly of $O\left(10^{-10}\right)$ when $N = 100$, as shown by Figure \ref{fig:4}, we expect the calculated value $J_{100}$ to be accurate to at least the shown number of figures. The figure also clearly shows that the OC variable is a bang-bang control with a transition as observed earlier in \cite{azzato2008applying}. Figure \ref{fig:Fig8} shows the convergence of $J_N$ to the computed value $-1.311$ for increasing values of $N$. It is interesting to mention here that when the fractional order $\alpha \to 1$ from the right, the discrete NLP becomes a steady state optimization problem associated with the approximate optimal performance index $J_{100} \approx 0$ and the approximate optimal states and control variables $(\bmx,\bmu) = \F{O}_{1,3}$, as manifested by the FGPS method through Figure \ref{fig:Fig6}. Figure \ref{fig:Fig7} shows the evolution of the approximate optimal state and control variables with the fractional order $\alpha$ for some increasing values of $\alpha$ approaching $1$.

\begin{figure}
\centering
\includegraphics[scale=0.24]{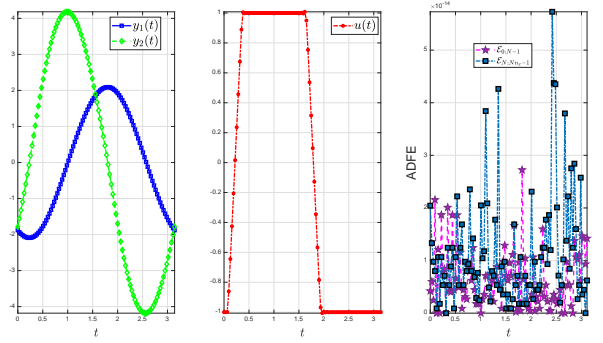}
\caption{The profiles of the approximate optimal state and control variables and the ADFEs obtained at the collocation nodes set $\MBS_{100}^{\pi}$ using the FGPS method together with MATLAB fmincon solver with the parameter values $N = 100, \alpha = 0.99999$, and $L = 30$. The plots of the state and control variables were generated using $100$ linearly spaced nodes in $\FOmega_{\pi}$.}
\label{fig:Fig5}
\end{figure}

\begin{figure}
\centering
\includegraphics[scale=0.5]{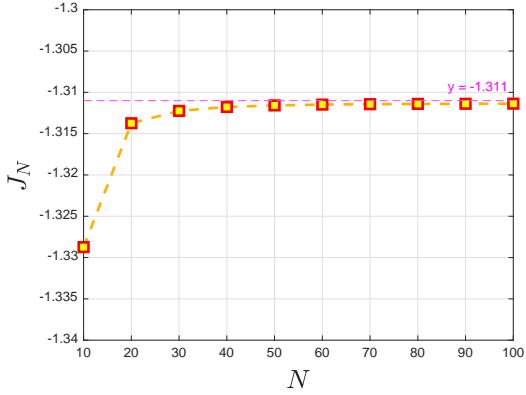}
\caption{The approximate optimal performance index values against $N$ for $N = 10$:$10$:$100, \alpha = 0.99999$, and $L = 30$.}
\label{fig:Fig8}
\end{figure}

\begin{figure}
\centering
\includegraphics[scale=0.24]{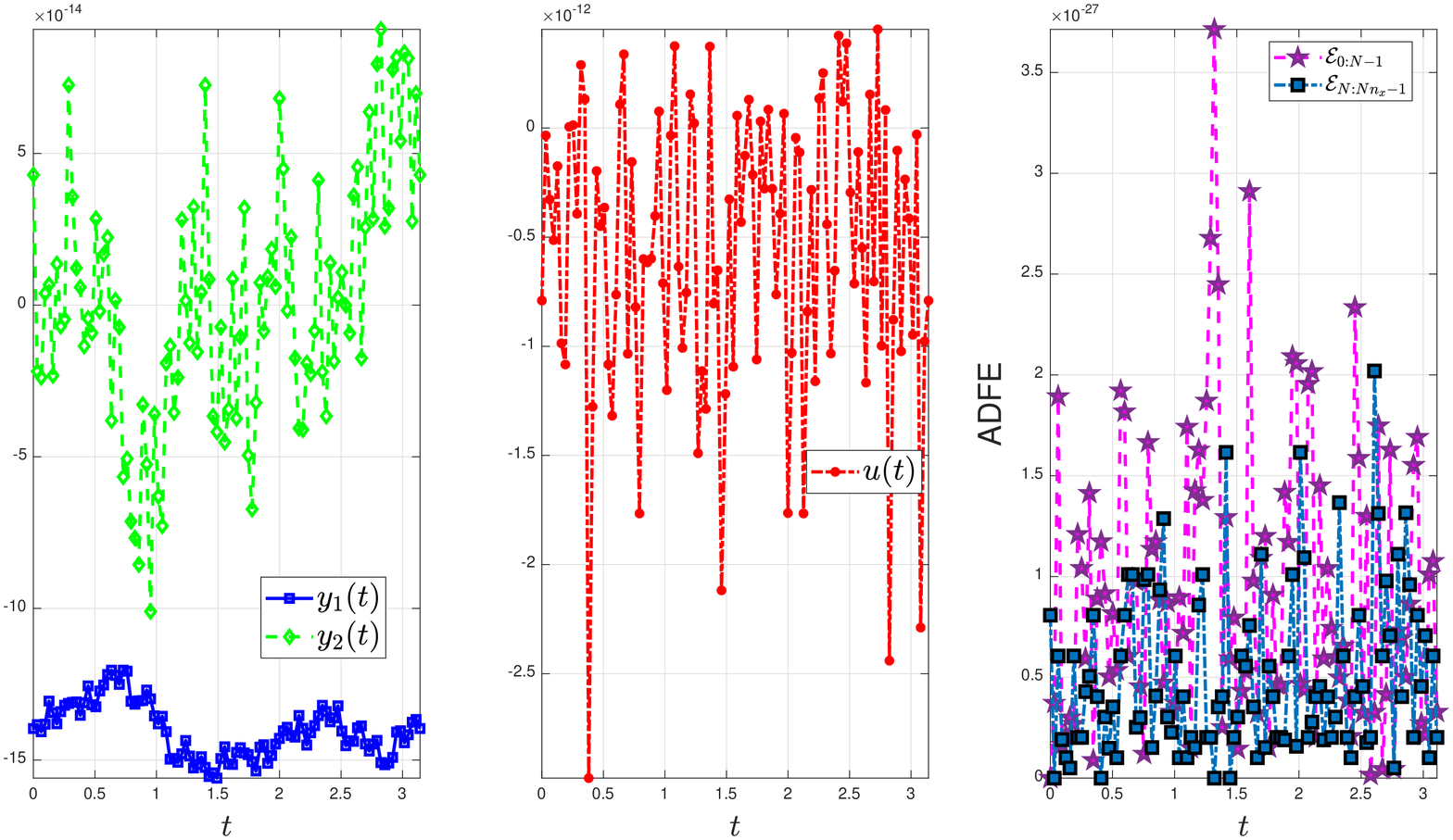}
\caption{The profiles of the approximate optimal state and control variables and the ADFEs obtained at the collocation nodes set $\MBS_{100}^{\pi}$ using the FGPS method together with MATLAB fmincon solver with the parameter values $N = 100, \alpha = 1.00001$, and $L = 30$. The plots of the state and control variables were generated using $100$ linearly spaced nodes in $\FOmega_{\pi}$.}
\label{fig:Fig6}
\end{figure}

\begin{figure}
\centering
\includegraphics[scale=0.24]{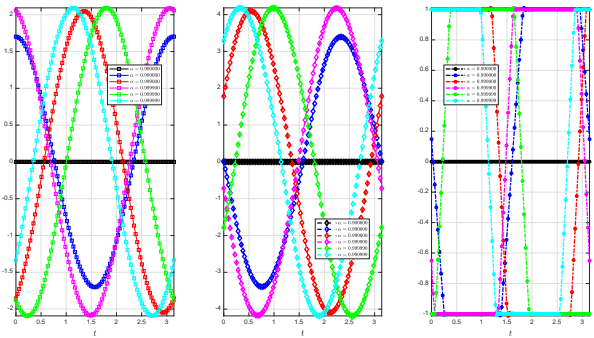}
\caption{The evolution of the approximate optimal state and control variables $y_1$ (left), $y_2$ (middle), and $u$ (right) for $\alpha = 0.9, 0.99, 0.999, 0.9999, 0.99999, 0.999999$ obtained using the FGPS-fmincon method with the parameter values $N = 100$ and $L = 30$. The plots were generated using $100$ linearly spaced nodes in $\FOmega_{\pi}$.}
\label{fig:Fig7}
\end{figure}

\section{Conclusion}
\label{sec:Conc}
This study introduced an extension of the earlier work we presented in \cite{elgindy2023fourier} to PHFOCPs with fractional order $\alpha \in \MBRP\backslash\MBZP$. The smart change of variables formula \eqref{eq:Elg1} largely simplified the problem of calculating the periodic FDs of periodic functions to the problem of evaluating the integral of the $m$th-derivatives of their trigonometric Lagrange interpolating polynomials. This strategy allows for the accurate computation of the singular integral of the FD formula by removing the singularity at the integration lower limit prior to numerical integration and renders the reduced integral well behaved and can be treated accurately and efficiently using Gegenbauer quadratures. We extended the notions of the FGPSQ and the FGPSFIM with index $L$, which proved to be very effective in computing periodic FDs, to exist for any fractional order $\alpha \in \MBRP\backslash\MBZP$. The derived Theorems \ref{sec:erranalysgp1} and \ref{thm:Jan212022} are very useful as they proved the exponential convergence of the FGPS expansion for sufficiently smooth periodic functions as $N_G \to \infty$, and allowed for the prediction of the quality of the FGPS approximations to FDs in light of some of the main parameters associated with the periodic FD and the proposed FGPS method. The numerical results of the benchmark PHFOCP demonstrated the accuracy, efficiency, and stability of the proposed FGPS method.

\section*{Declarations}
\subsection*{Competing Interests}
The author declares there is no conflict of interests.

\subsection*{Availability of Supporting Data}
The author declares that the data supporting the findings of this study are available within the article.

\subsection*{Ethical Approval and Consent to Participate and Publish}
Not Applicable.

\subsection*{Human and Animal Ethics}
Not Applicable.

\subsection*{Consent for Publication}
Not Applicable.

\subsection*{Funding}
The author received no financial support for the research, authorship, and/or publication of this article.

\subsection*{Authors' Contributions}
The author confirms sole responsibility for the following: study conception and design, data collection, analysis and interpretation of results, and manuscript preparation.

\bibliographystyle{model1-num-names}
\bibliography{Bib}
\end{document}

%% file: Tit.tex
\title{Fourier-Gegenbauer Pseudospectral Method for Solving Periodic Higher-Order Fractional Optimal Control Problems}
\author[Assiut]{Kareem T. Elgindy\corref{cor1}}
\ead{kareem.elgindy@gmail.com}
\address[Assiut]{Mathematics Department, Faculty of Science, Assiut University, Assiut 71516, Egypt}
\cortext[cor1]{Corresponding author}

%% file: Abs.tex
\begin{abstract}
In \cite{elgindy2023fourier}, we inaugurated a new area of optimal control (OC) theory that we called ``periodic fractional OC theory,'' which was developed to find optimal ways to periodically control a fractional dynamic system. The typical mathematical formulation in this area includes the class of periodic fractional OC problems (PFOCPs), which can be accurately solved numerically for a fractional order $\alpha$ in the range $0 < \alpha < 1$ using Fourier collocation at equally spaced nodes and Fourier and Gegenbauer quadratures. In this study, we extend this earlier work to cover periodic higher-order fractional OC problems (PHFOCPs) of any positive non-integer fractional order $\alpha$.
\end{abstract}
\begin{keyword}
Fourier collocation; Fractional optimal control; Gegenbauer quadrature; Periodic fractional derivative; Pseudospectral method.
\end{keyword}